\documentclass[american, 12pt]{amsart}
\usepackage{pb-diagram, srcltx}
\usepackage{amsfonts,amssymb}
\usepackage{amscd}
\usepackage{euscript}
\usepackage{hyperref}
\usepackage{fancyhdr}
\usepackage{float}
\usepackage[left=2.5cm,right=2cm, bottom= 3cm]{geometry}
\usepackage{subfigure}
\usepackage[utf8]{inputenc}
\usepackage{enumerate}
\usepackage[curve]{xypic}
\usepackage{graphicx}
\usepackage{xcolor}

\usepackage{color}

\newcounter{warn}[page]
\newcommand{\danger}{${\color{red}\triangle}\llap{\raisebox{.3ex}%
{\tiny!\hspace{1.45ex}}}$}
\newcommand{\warning}[1]{%
\raisebox{.01em}[0em]{\danger\ifnum\value{warn} > 1%
\tiny\bf\arabic{warn}\fi}%
\marginpar{\color{red}\tiny{\ifnum\value{warn} > 1\tiny\bf\arabic{warn}:\fi}\tiny #1}%
\stepcounter{warn}}

\newtheorem{theorem}{Theorem}
\newtheorem{definition}[theorem]{Definition}

\newtheorem{corollary}[theorem]{Corollary}
\newtheorem{proposition}[theorem]{Proposition}

\newtheorem{example}[theorem]{Example}
\theoremstyle{remark}
\newtheorem{remark}{Remark}

\renewcommand{\emptyset}{\varnothing}

\newcommand{\Z}{\mathbb{Z}}

\newcommand{\R}{\mathbb{R}}
\newcommand{\C}{\mathbb{C}}

\newcommand{\dd}{\mathrm{d}}

\begin{document}

\title[A note on Lagrangian intersections and Legendrian Cobordism]{A note on Lagrangian intersections \\and Legendrian Cobordism.}
  
\author{Lara Simone Su\'arez}
\address{Ruhr-Universität Bochum, Fakultät für Mathematik, Universitätsstraße 150, 44801 Bochum, Germany}
\email{lara.SuarezLopez@rub.de}




\begin{abstract} Let $\Lambda, \Lambda'$ be a pair of closed Legendrian submanifolds in a closed contact manifold $(Y, \xi = Ker(\alpha))$ related by a Legendrian cobordism $W\subset (\C\times Y, \tilde{\xi}=Ker(-y \d dx +\alpha))$. In this note, we show that in the hypertight setting, if $\Lambda$ intersects a closed, weakly exact or monotone pre-Lagrangian $P\subset Y$ for reasons of Floer homology, then so does $\Lambda'$.
\end{abstract}
\maketitle
\section{Introduction}
Let $(Y,\xi)$ be a co-oriented closed contact manifold. 
In this paper we consider pairs $(P,\Lambda)$ consisting of a closed Legendrian submanifold $\Lambda \subset Y$ and a closed pre-Lagrangian $P\subset Y$ (see Definition \ref{DefPreLag}). A usual question in symplectic topology is that of the displaceability of Lagrangian submanifolds by Hamiltonian isotopy. In contact topology, this question can be interpreted as the search for pairs $(P, \Lambda)$ with the so called \textit{intersection property} as defined by Eliashberg-Polterovich in \cite{ElPo}.
A pair $(P,\Lambda)$ of $(Y,\xi)$ has \textit{the intersection property}, if for every contactomorphism $\phi \in Cont_{0}(Y,\xi)$\footnote{The $0$ here stands for the connected component of the identity.} the intersection $P\cap \phi(\Lambda)$ is non-empty. 
A tool to find such pairs was introduced by Eliashberg-Hofer-Salamon in \cite{EHS}. It is a Floer homology group $HF(P,\Lambda;\Z_2)$ for pairs $(P, \Lambda)$, that is invariant under Legendrian isotopy. In some special cases, this group is isomorphic to the homology of the Legendrian $\Lambda$. This is the case in the following examples from \cite{EHS}, which are the first examples of pairs with the intersection property:
\begin{enumerate}
    \item Let $X$ be a closed manifold and $P^{+}(T^*X)$ the space of co-oriented contact elements with the canonical contact structure. If there is a non-singular closed one form $\beta$ on $X$ then there is a pre-Lagrangian $P$ associated to $ \text{graph}(\beta)$. If moreover, $P$ is foliated by closed Legendrians and $\Lambda$ is any Legendrian leaf, then the pair $(P, \Lambda)$ has the intersection property.
    \item Let $(M,\omega)$ be a closed symplectic manifold with $[\omega]\in H^{2}(M;\Z)$. In a prequantization space $QM$ of $(M,\omega)$, if there is a closed Lagrangian submanifold $L\subset(M,\omega)$ that satisfies the Bohr-Sommerfeld condition and $\pi_2(M,L)=0$ then the corresponding pre-Lagrangian $P$ in $QM$ and a flat Legendrian lift $\Lambda$ of $L$ is a pair with the intersection property. This last example was independently  discovered by Ono \cite{On}.
\end{enumerate}

In \cite{ElPo} Eliashberg-Polterovich also defined the \textit{stable} intersection property of a pair, meaning that in the stabilized setting $(T^*S^1\times Y, Ker(rdt +\alpha))$ the pair $(S^1 \times \{0\}\times P, S^1\times \{0\}\times \Lambda)$ has the intersection property. They showed that the previous examples have the stable intersection property. 
The (stable) intersection property is related to the orderability of the contact manifold. In fact in \cite[Theorem 2.3.A]{ElPo} they proved that having a pair with this property implies that the manifold is orderable.

From another perspective, we can say that in the previous examples, the stabilization preserves the intersection property. In the same way, in this note we remark, in particular for the above examples, that the intersection property is preserved by Legendrian cobordism with the following property: there is a contact form with no contractible Reeb orbits or chords with boundary on the cobordism. This is the \textit{hypertight} property defined below in Definition 2.

The notion of Legendrian cobordism was introduced by Arnol'd (\cite{Ar1}, \cite{Ar2}) who computed the immersed Legendrian cobordism group (oriented and not-oriented) when the immersed Legendrian has dimension $n=1$, the cases $n\geq 2$ were computed by Audin \cite{Aud} and Eliashberg \cite{El} independently.  Afterwards, related works on Legendrian cobordisms where done by Ferrand \cite{Fe} and more recently by Limouzineau \cite{Lim}.

An embedded Legendrian cobordism in $J^1(\R\times N)$, projects to an immersed exact Lagrangian cobordism in $T^*N$. Biran-Cornea \cite{BiCo} showed that monotone embedded Lagrangian cobordism preserves Floer-theoretical Lagrangian invariants after Chekanov \cite{Che} showed how these preserve certain counts of Maslov two disks. 

In general, embedded Legendrian cobordism does not preserve holomorphic curves type invariants. This is the case, for instance of the linearized contact homology, a Legendrian analogue of the Lagrangian Floer homology. In $(\R^3,dz-pdq)$ for example, it follows from the work of Arnol'd that two Legendrian knots are oriented Legendrian cobordant if and only if they have the same Maslov index. Hence any two knots with same Maslov index but different linearized contact homology are still Legendrian cobordant.

However, under additional restrictions, we show that Legendrian cobordism does preserve the Floer homology group of Eliashberg-Hofer-Salamon. More precisely we prove the following:

\begin{theorem}
Let $(Y,\xi)$ be a closed hypertight contact manifold and $\Lambda, \Lambda'\subset Y$ a pair of hypertight Legendrian submanifolds related by a connected hypertight Legendrian cobordism $(W;\Lambda, \Lambda')\subset (\C \times Y,\tilde{\xi})$.  If $P\subset Y$ is a closed weakly exact or monotone  with $N_P \geq 2$ pre-Lagrangian submanifold 
then $$HF(P,\Lambda) \cong  HF(P,\Lambda').$$
\end{theorem}

The proof of Theorem 1 is an adaptation to the cobordism setting of \cite[Theorem 3.7.3]{EHS}. We define a Floer complex for the Lagrangian lift of a Legendrian cobordism $(W;\Lambda,\Lambda')$ and a suitable cylindrical Lagrangian $\tilde{P}$ obtained from $P$. We then observe that the homology of this complex is isomorphic to the homology of the Floer complex of the Legendrian boundary $\Lambda$ (or $\Lambda'$) and the pre-Lagrangian $P$. This proof uses a similar strategy to the one used by Biran-Cornea \cite{BiCo} to show that monotone Lagrangian cobordism preserves Floer homology. Combining Theorem 1 with \cite[Theorem 2.5.1-2.5.4]{EHS} (see section 2.2 Theorem 11) we get the following corollary.
\begin{corollary}
Under assumptions in Theorem 1, consider a pair $(P, (W;\Lambda, \Lambda'))$ where $P$ is weakly exact and $\Lambda \subset P$.  If moreover boundary homomorphism $\pi_{2}(Y,P)\to  \pi_1(P)$ is trivial, then if $\{\phi_t\}_{0\leq t\leq 1}$ is a contact isotopy such that $P \pitchfork \phi^1(\Lambda')$. Then  $$\#\phi_{1}(\Lambda') \cap P \geq  rank(H_*(\Lambda, \Z_2)).$$
\end{corollary}
\textit{Acknowledgements.}
I thank Egor Shelukhin for suggesting this project and helpful discussions. This research is supported by the Floer Centre of Geometry at Ruhr-University Bochum and is part of a project in the SFB/TRR 191 \textit{Symplectic Structures in Geometry, Algebra and Dynamics}, funded by the DFG.

\section{Setting}
Let $(Y,\xi)$ be a closed co-oriented contact $2n+1$-manifold. Denote by $\text{Cont}^{+}(\xi)$ the set of contact forms defining the same co-orientation on $\xi$.

For a choice of $\alpha\in\text{Cont}^{+}(\xi)$, the symplectic manifold $(S_{\alpha}(Y),\omega):=(\R\times Y, \dd(e^\theta \alpha))$ is called the \textit{symplectization} of $(Y, \alpha)$.

Given $\eta\in \text{Cont}^{+}(\xi)$, its \textit{Reeb vector field} is the unique vector field $R_{\eta}\in \Gamma(TY)$ satisfying the two equations: $$\iota_{R_{\eta}}\eta = 1,$$ 
$$\iota_{R_{\eta}}\dd\eta = 0.$$ Its associated flow is denoted by $\{\phi^t_{R_{\eta}}\}$ and it is called the \textit{Reeb flow}.

\begin{definition} A contact manifold $(Y, \xi)$ that admits a contact form with no contractible periodic Reeb orbits is called \textbf{hypertight}. Such a contact form is called a \textbf{hypertight contact form}. A Legendrian submanifold $\Lambda\subset Y$ for which there is a hypertight contact form such that any Reeb chord with boundary on $\Lambda$ represents a non-trivial class in $\pi_{1}(Y, \Lambda)$ is called a \textbf{hypertight Legendrian} submanifold.
\end{definition}
Some examples of hypertight contact manifolds are jet spaces, certain pre-quantization spaces and certain unit cotangent bundles, with their corresponding standard contact structures.

In this note \textbf{all} contact manifolds and Legendrians are assumed to be hypertight.

\begin{definition}\label{DefPreLag}A \textbf{pre-Lagrangian} is a $(n+1)$-dimensional submanifold $P\subset Y^{2n+1}$ for which there exists a contact form $\alpha \in Cont^{+}(\xi)$ such that $\dd(\alpha\vert_{P}) = 0$ i.e $\alpha\vert_{P}$ is closed.
\end{definition} 
If there is a function $f:P \rightarrow \R$ such that $\alpha|_{P} = \dd f$ then $P$ is called a \textit{exact pre-Lagrangian} and $f$ is called \textit{contact potential} on $P$.
An exact Lagrangian lift $\hat{P}$ of an exact pre-Lagrangian is given by $$\hat{P}=\{0\}\times P\subset \R\times Y. $$
Note that $\hat{f}:\hat{P}\rightarrow \R$, $\hat{f}(0,p)=f(p)$ is a primitive for $e^{\theta}\alpha|_{\hat{P}}$.

A pre-Lagrangian  $P$ is called \textit{weakly exact} when $\dd \alpha|_{\pi_2(Y,P)}=0$ and \textit{monotone} if there is a positive constant $K$ such that for all $u\in \pi_2(Y,P)$ we have $$\int_{\partial u} \alpha = K \mu(u)$$ where $\mu$ denotes the Maslov class $\mu:\pi_2(Y,P)\to\Z$. The positive generator of the image of the homomorphism defined by $\mu$ is denoted $N_P$ and is called the minimal Maslov number of $P$.
\begin{example}Let $(Y,\xi) = (S(T^*X),\text{ Ker}(p\dd q|_{S(T^*X)}))$ where the right hand side is the unit cotangent bundle of $(X,g)$ with respect to some choice of a Riemannian metric $g$. Let $\eta$ be a nowhere vanishing closed 1-form. If $\Gamma_{\eta}\subset T^*X$ denotes the graph of $\eta$, then $\Gamma_{\eta/||\eta||} \subset S(T^*X)$ is a pre-Lagrangian for the contact form $||\eta||p\dd q$.
\end{example}
\begin{example}Let $(Y,\xi) = QX$ be the prequantization of a symplectic manifold $(X,\omega)$. Topologically $Y$ is an $S^1$-principal bundle over $X$ with projection map $p:Y\rightarrow X$. Given a Lagrangian submanifold $L\subset X$ its lift  $p^{-1}(L)\subset(Y,\xi)$ is a pre-Lagrangian.
\end{example}

\subsection{Legendrian cobordism.}
Let $(x,y)$ denote coordinates on $\C$. Then $\tilde{\alpha}= -y\dd x + \alpha$ is a contact form on $\C\times Y$. Denote the resulting contact manifold by  $$(\tilde{Y}, \tilde{\xi})= (\C \times Y, Ker(\tilde{\alpha})).$$ Let $\pi_{\C}: \C\times Y \rightarrow \C$ denote the projection map. Given $\Lambda, \Lambda'\subset Y$ two closed Legendrian submanifolds, in \cite{Ar1} and \cite{Ar2} Arnol'd introduced the notion of Legendrian cobordism on which the following definition is based.
\begin{definition}{A \textbf{Legendrian cobordism} $(W; \Lambda, \Lambda')$ is an embedded Legendrian submanifold $W\subset (\tilde{Y},\tilde{\xi})$ such that for some $0 < \epsilon \in \R$  and $R\geq 1$ we have that $$W \cap \pi_{\C}^{-1}([\epsilon, R-\epsilon] \times \R)$$ is a smooth compact manifold with boundary $\Lambda \sqcup \Lambda'$ and $$W \cap \pi_{\C}^{-1}(\C \setminus ([\epsilon, R-\epsilon] \times \R)) = (-\infty, \epsilon)\times \{0\} \times \Lambda\sqcup (R-\epsilon, \infty)\times \{0\} \times \Lambda'.$$}
\end{definition}

\begin{example}(Legendrian suspension)\cite{ElPo} Let $H: \R \times Y \rightarrow \R$ be a contact Hamiltonian with $H_{t}\equiv0$ for $t\leq 0$ and $t\geq 1$. The associated contact vector field $X_{H}$ is given by the two conditions: $\iota_{X_{H}}\alpha = H$ and $ \iota_{X_{H}}\dd \alpha = -\dd H + (\iota_{R_{\alpha}}\dd H )\alpha$.
Denote by $\{\psi^{t}\}$  the contact isotopy associated to $X_{H}$. Let $\Lambda \subset M$ be a compact Legendrian.
The map:
\begin{align*}
\Phi: \R \times\Lambda  &\rightarrow \C \times Y \\ (t,p) &\mapsto (t, H(t,\psi^{t}(p)), \psi^{t}(p)),
\end{align*}
is a Legendrian embedding into $(\C \times M, Ker(\tilde{\alpha}))$ defining a Legendrian cobordism $$(\Phi(\R\times \Lambda);\Lambda, \psi^{1}(\Lambda)).$$
\end{example}
\begin{remark}Any Legendrian isotopy  defines a Legendrian cobordism given by the Legendrian suspension of a contact Hamiltonian generating the isotopy. 
\end{remark}

\begin{example}The trace of surgery: Elementary Legendrian cobordisms can be constructed using the Lagrangian handle in Haug \cite{Hau} building on the work of Dimitroglou Rizell \cite{Di}. This construction produces an embedded Legendrian cobordism in $J^{1}(\R^{n+1})$. 

The Legendrian $k$-handle: Consider the non compact Legendrian $W_{\epsilon, k}\subset J^{1}(\R^{n+1})$ given by $$W_{\epsilon, k}= \{(x_{0},\textbf{x}, \pm \dd F(x_{0},\textbf{x}), \pm F(x_{0}, \textbf{x}))\in J^{1}(\R^{n+1}) \hspace{0.2cm}\vert\hspace{0.2cm} (x_{0},\textbf{x}) \in \mathcal{U}\}$$ for $\textbf{x}= (x_{1}, ...,x_{n}) \in \R^{n}$ where $F(x_{0}, \textbf{x})= (f(x_{0}, \textbf{x}))^{\frac{3}{2}}$,  $$f(x_{0}, \textbf{x})=\sum \limits_{i=1}^{k} x_{i}^{2} -  \sum \limits_{i=k+1}^{n} x_{i}^{2} + \sigma(\sum\limits_{i=1}^{k} x_{i}^{2})\rho(x_{0}) -1$$
and $\mathcal{U}= \{(x_{0}, \textbf{x})\hspace{0.2cm}\vert\hspace{0.2cm} f(x_{0}, \textbf{x}) \geq 0\}$, where the functions $\sigma$ and $\rho$ look like:

\hspace{2cm}
\includegraphics[scale=0.5]{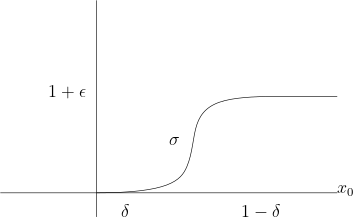}\hspace{2cm}
\includegraphics[scale=0.5]{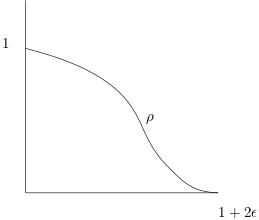}\\
In the case a Legendrian $\Lambda$ contains a $l$-sphere $S^l\subset \Lambda$ which bounds an isotropic $l+1$-disk $D^{l+1} \subset Y\setminus \Lambda$, compatible with the Legendrian in some sense made precise in \cite{Hau}, after surgery a new Legendrian is obtained, and an embedded Legendrian cobordism between the initial Legendrian and the one produced after the surgery exists.
\end{example}
\begin{example}Let $(M, \lambda)$ be a Liouville domain. The Legendrian lift of an immersed exact Lagrangian cobordism $(W;L_0, L_1) \subset (\C \times M, \dd (-y\dd x + \lambda))$ to the contactization of $\C\times M$, defined by $\C\times C(M,\lambda)=(\C\times M\times \R_z, -y\dd x + \lambda+ \dd z )$. More precisely, if $f:W \to \R$ is a potential for the Lagrangian $W$, the Lagrangian $\tilde{W}:=\{(w,f(w)) | w\in W\}$ is an embedded Legendrian cobordism in $\C \times C(M,\lambda)$.
\end{example}


\subsection{Previous results}
We denote by $HF(P,\Lambda;\Z_{2})$ the homology of the complex defined by Eliashberg-Hofer-Salamon in \cite{EHS}. In  \cite[Theorem 3.7.3]{EHS} they showed that in the hypertight setting, this homology group is well defined and invariant under compactly supported Hamiltonian isotopy. 
Let $(Y,\xi)$ be as in Examples (1) or (2) in the introduction, namely, either (1) a sphere cotangent bundle of a compact manifold or (2) a prequantization space. Assume that the contact manifold and pre-Lagrangian submanifold satisfy the topological condition that the boundary map $\pi_2(Y,P)\to \pi_1(P)$ is trivial. In this setting, when the pair $(P,\Lambda)$ is either the graph of a non-zero closed form foliated by closed Legendrians and a Legendrian leaf of it in the first case, or the pre-Lagrangian lift of a Lagrangian satisfying the Bohr-Sommerfeld condition with the property that $\pi_2(Y,P)=0$ and a Legendrian lift of it, then the following theorem holds.
\begin{theorem}{\cite{EHS}[1995]\label{thm=EHS}
Let $\{\phi_t\}_{0\leq t\leq 1}$ be a contact isotopy and $\Lambda$ a Legendrian such that $P \pitchfork \phi^1(\Lambda)$. Then  $$\#\phi_{1}(\Lambda) \cap P \geq  rank(H_*(\Lambda, \Z_2)).$$
}
\end{theorem}
In the same direction Akaho \cite{Ak}, proved a version of the previous theorem replacing the hypertight condition by a small energy condition.
Results about displaceability of pre-Lagrangians can also be found in the work of Marinkovic-Pabiniak \cite{MaPa}.

\section{A Floer complex for the pair $(P, W)$.}
The version of Floer complex associated to the pair $(P,W)$, where $P \subset Y$ is a weakly exact or monotone pre-Lagrangian (with $N_P\geq 2$) and $(W,\Lambda, \Lambda')\subset \C\times Y$ is a Legendrian cobordism, is an adaptation to the cobordism setting (following Biran and Cornea in \cite{BiCo}) of the Floer complex as defined in \cite{EHS}. 

Let $\tilde{\alpha} \in Cont^+(\tilde{\xi})$ be a contact form on $\C \times Y$ that can be written as $\tilde{\alpha}= -y\dd x+\alpha$ where $\alpha\in Cont^+(\xi)$. We assume that $\tilde{\alpha}$ satisfies:
\begin{itemize}
\item $P$ is a pre-Lagrangian for $\alpha$, i.e $d\alpha|_{P}=0$.
\item $W$ is hypertight for $\tilde{\alpha}$.
\end{itemize}
To the pair $(P,W)$ we associate the pair $(L_P,L_W)$ in $(S(\tilde{Y}), \tilde{\omega} =\dd (e^{\theta}\tilde{\alpha}))$, where $L_{P}=\{0\}\times \R\times\{0\}\times P$ denotes the Lagrangian lift of the pre-Lagrangian $\tilde{P}=\R\times\{0\}\times P \subset \C\times Y$ and $L_{W} = \R\times W$ is the Lagrangian lift of $W$. 

\subsection{Choices}\label{chice}
The Floer complex will depend on the choice of triples $(H,f,J)$ where:
\begin{enumerate}
    \item Denote by $\pi_{\C}:\C \times Y\rightarrow \C$ the projection. Let us fix a compact set $B\subset \C$ of the form $[R_-,R_+] \times [R_-,R_+]$ for $R_{\pm}\in \R$ big enough, such that $\pi_{\C}(W) \cap \C\setminus B$ is $$(-\infty,R_-]\times \{0\} \cup [R_+, \infty)\times \{0\}.$$
    \item A Hamiltoninan $H:[0,1]\times S(\tilde{Y})\rightarrow \R $ compactly supported in $\R\times \pi^{-1}_{\C}(B)\subset S(\tilde{Y})$. We use the Hamiltonian flow of $H$ to perturb $L_P$ to make it transverse to $L_W$ on $\R\times \pi^{-1}_{\C}(B)\subset S(\tilde{Y})$. 
    \item A smooth function $f:\C \rightarrow  \R$  with the property that 
    $$f(x,y) = a_+ x + b_+ \text{ for } x<R_-  - \epsilon \text{ and }f(x,y) = a_- x + b_- \text{ for } x> R_+ + \epsilon$$ where $a_\pm, b_{\pm}\in \R$ and $\epsilon >0$. We consider the Hamiltonian $\tilde{f} = e^{\theta} (f\circ \pi_{\C})$ on $S(\tilde{Y})$. It is the Hamiltonian lift of the contact Hamiltonian $f\circ \pi_{\C}$ to $S(\tilde{Y})$. 
    The corresponding Hamiltonian flow $\psi_{\tilde{f}}^t:S(\tilde{Y})\rightarrow S(\tilde{Y})$ is given by $$(\theta,x,y,p)\in \R\times \C\times Y \mapsto \psi_{\tilde{f}}^t(\theta,x,t,p)= (\theta,x,y+a_{\pm}t,\phi_{R_{\alpha}}^{a_{\pm}xt}(p)),$$
    for $(\theta,x,y,p)$ outside $\R\times B\times Y$.
    This Hamiltonian flow perturbs the cylindrical ends of $L_{P}$ in such a way that the time one map $\psi^{1}_{\tilde{f}}(L_{P})\cap L_W = \emptyset$ outside $\R\times B \times Y$. All this while keeping the contact form invariant at infinity. 
    \item $J$ is an admissible $\tilde{\omega}$-compatible almost complex structure on $S(\tilde{Y})$. Admissible means a cylindrical and $\dd \tilde{\alpha}$-compatible almost complex structure for which $d\pi$ is $(J,j)$-holomorphic where $\pi: (S(\tilde{Y}), J)\to (\C,j)$ denotes the projection and $j$ denotes the complex structure on $\C$. We define admissible in more detail in the next section. \end{enumerate}

Let  
$$\mathbf{\Lambda} = \{\sum_{k=0}^{\infty} a_k T^{\lambda_k} \hspace{0.2cm}|\hspace{0.2cm} a_k \in \Z_{2}, \lambda_k \in \R, \lim_{k\rightarrow\infty}\lambda_{k}=\infty\}. $$
The $\star$-Floer complex of the pair $(L_{P},L_W)$ is denoted by
\begin{equation}\label{eq:Floer complex}
CF_{\star}(L_P,L_W; H,f,J):= (\mathbf{\Lambda}\langle \hat{\mathcal{I}}_{\star}(\psi^{1}_{\tilde{f}}(L_P),L_W; H)\rangle, d_J).
\end{equation}
The set $\hat{\mathcal{I}}_{\star}(\psi^{1}_{\tilde{f}}(L_{P}),L_W; H)$ is constructed as follows:

Consider the set $\mathcal{P}(\psi^{1}_{\tilde{f}}(L_P), L_W) := \{\gamma\in C^{0}([0,1];S(\tilde{Y})) \hspace{0.2cm}|\hspace{0.2cm} \gamma(0)\in \psi^{1}_{\tilde{f}}(L_P), \gamma(1)\in L_W\}$ of paths between $\psi^{1}_{\tilde{f}}(L_P)$ and $L_W$. For a fix $* \in \psi^{1}_{\tilde{f}}(L_P)\cap L_W$ denote by $\star=[*]\in \pi_{0}(\mathcal{P}(\psi^{1}_{\tilde{f}}(L_P), L_W))$ and let $\mathcal{P}_{\star}(\psi^{1}_{\tilde{f}}(L_P), L_W)$ be the connected component corresponding to the class $\star$. When $\psi^{1}_{\tilde{f}}(L_P)\cap L_W$ is not transverse, we use the Hamiltonian $H$. Denote by $$\mathcal{I}_{\star}(\psi^{1}_{\tilde{f}}(L_P),L_W; H) := \{ x\in  \psi^1_H(\psi^{1}_{\tilde{f}}(L_P))\cap L_W \hspace{0.2cm}|\hspace{0.2cm} x\in \mathcal{P}_{\star}(\psi^{1}_H(\psi^{1}_{\tilde{f}}(L_P)), L_W)\}$$ the intersection points in the connected component of $\star$. Then
\begin{align*}
    &\hat{\mathcal{I}}_{\star}(\psi^{1}_{\tilde{f}}(L_P),L_W; H) :=\\
    &\{(x,\tilde{x}) \in \mathcal{I}_{\star}(\phi^{1}_{\tilde{f}}(L_P),L_W; H) \times C^0([0, 1],\mathcal{P}_{\star}\Big(\psi^{1}_H(\psi^{1}_{\tilde{f}}(L_P)), L_W\Big)) | \tilde{x}(0)= x, \tilde{x}(1)= * \}/\sim
\end{align*}
where $(x,\tilde{x})\sim (x',\tilde{x}')$ iff $x = x'$ and $\mu(\tilde{x}\#(\overline{\tilde{x}'}))=0$ where $\mu$ denotes the Maslov index and $\overline{\tilde{x}'}$ denotes the same map with opposite orientation. Then the elements of $\hat{\mathcal{I}}_{\star}(\psi^{1}_{\tilde{f}}(L_P),L_W; H)$ have a well defined index $|[(x,\tilde{x})]| = \mu(\tilde{x}) \in \Z$.

 By choosing suitable coefficients $a_\pm$ for the perturbation $f$ we can ensure that $\psi^{1}_{\tilde{f}}(L_P)\cap L_W$ is contained in a bounded region. Then perturbing by a generic, compactly supported Hamiltonian $H$ will ensure that $ \psi^1_H(\psi^{1}_{\tilde{f}}(L_P)) \cap L_W$ is a finite set and the intersection is transverse making the set $\mathcal{I}_{\star}(\psi^{1}_{\tilde{f}}(L_P),L_W; H)$ finite.
 
 Let $\mathbf{J}=\{J_t\}_{t\in [0,1]} $ be a time dependent almost complex structure and let $\hat{x}_{\pm}= [(x_{\pm}, \tilde{x}_{\pm})] \in \hat{\mathcal{I}}_{\star}(\psi^{1}_{\tilde{f}}(L_P),L_W; H)$. We define moduli spaces
 \begin{equation*}
 \mathcal{M}(\hat{x}_{-}, \hat{x}_{+};\mathbf{J}):=\Big\{ u\in C^{1}([0,1]\times \R; S(\tilde{Y})) \hspace{0.2cm}|\hspace{0.2cm}
 \substack{\overline{\partial}_{\mathbf{J}}u=0, E(u) <\infty\\ \lim\limits_{s\to \pm \infty}u(t,s)= x_{\pm}(t),\\
 u(0,s) \in \psi^{1}_{H}(\psi^1_{\tilde{f}}(L_P)), u(1,s)\in L_W\\
   \mu(\tilde{x}_{-}) = \mu(u\#\tilde{x}_{+})} \Big\}/\R.
 \end{equation*} 

\subsubsection{Admissible almost complex structure}
\begin{definition}
An almost complex structure $J$ on $S(\tilde{Y})=\R\times\tilde{Y}$ is called cylindrical if $J$ is $\R$-invariant and $J \partial_{\theta}\in T\tilde{Y}$.
\end{definition}
Let $J$ be an almost complex structure on $\tilde{\xi}$ compatible with the two form $d\tilde{\alpha}$. By setting $J \partial_{\theta} = R_{\tilde{\alpha}}$, $J$ extends to a cylindrical almost complex structure $\tilde{J}_{\tilde{\alpha}}$ on $\R\times \tilde{Y}$.
For $(\hat{a}, \hat{m}) \in T_{(a,m)}\R\times \tilde{Y}$, $\tilde{J}_{\tilde{\alpha}}$ is given by 
\begin{equation}\label{J-Cyl}
(\tilde{J})_{(a,m)}(\hat{a}, \hat{m}) = (-\tilde{\alpha}_{(a,m)}(\hat{m}), J_{m}(\pi_{\tilde{\xi}}(\hat{m})) + \hat{a} R_{\tilde{\alpha}}),
\end{equation}
here $\pi_{\tilde{\xi}}: T(\R \times \tilde{Y}) \rightarrow \tilde{\xi}$ is the projection to $\tilde{\xi}$ parallel to $R_{\tilde{\alpha}}$. Such an almost complex structure is called \textit{compatible with the contact form $\tilde{\alpha}$}.

\begin{definition}\label{Def_admJ}
An almost complex structure $J$ on $S(\tilde{Y})$ is called admissible if:
\begin{enumerate}
\item[C1] $J$is cylindrical and restricts to a $d\tilde{\alpha}$-compatible complex structure on the plane bundle $\tilde{\xi}$.
\item[C2] There is a compact set $D \subset \R\times \C$  such that $D =D_0\times D_{1}$ with $B\subset D_{1}$ for $B$ the fixed compact set on section 3.1, choice 1. Moreover,
 $J\vert_{((\R\times \C) \setminus D)\times Y)} = \tilde{J}_{\tilde{\alpha}}$ where the right hand side is $\tilde{\alpha}$-compatible as defined above.
\item[C3] On $((\R\times \C) \setminus D)\times Y)$ the almost complex structure $J$ satisfies that the projection $\pi:(\R\times\C \times Y,J) \rightarrow (\C,i)$, is $(J, j)$-holomorphic.
\end{enumerate}
We denote the space of admissible almost complex structures by $\mathcal{J}_{ad}$.
\end{definition}
Notice that any $J$ compatible almost complex structure on $\xi$ extends to an admissible one. Indeed, since $\tilde{\xi}_{( x, y,p)} = \R\langle\partial_x + y R_{\alpha}\rangle\oplus \R\langle\partial_y \rangle \oplus \xi_p$ then $J$ extends to $\tilde{\xi}$ by setting $J_{(x,y,p)}(\partial_x + y R_{\alpha})= \partial_y$ and $J_{(x,y,p)}(\partial_y)= -(\partial_x + y R_{\alpha})$ and then $\tilde{J}_{\tilde{\alpha}}\in \mathcal{J}_{ad}$ is admissible.
\begin{proposition}
Let $\mathbf{J}=\{J_t\}_{t\in [0,1]} $ be a time dependent family of generic and admissible almost complex structures and let $\hat{x}_{\pm}= [(x_{\pm}, \tilde{x}_{\pm})] \in \hat{\mathcal{I}}_{\star}(\psi^{1}_{\tilde{f}}(L_P),L_W; H)$. The moduli space $\mathcal{M}(\hat{x}_{-}, \hat{x}_{+};\mathbf{J})$ is a manifold of dimension $(|x_-| -|x_+|-1)$. If $(|x_-| -|x_+|-1) = 0$ then $\mathcal{M}(\hat{x}_{-}, \hat{x}_{+};\mathbf{J})$ is a compact manifold. If $(|x_-| -|x_+|-1) = 1$ then $\mathcal{M}(\hat{x}_{-}, \hat{x}_{+};\mathbf{J})$ admits a compactification by gluing broken strips. 
\end{proposition}
\begin{proof}
The first claim about transversality of the moduli spaces $\mathcal{M}(\hat{x}_{-}, \hat{x}_{+};\mathbf{J})$ for generic choice of $\mathbf{J}$ goes back to Floer and can be found in Floer \cite[Theorem 4a]{Fl}. 
The proof of compactness is a combination of arguments in Biran-Cornea \cite[Lemma 4.2.1]{BiCo} and Eliashberg-Hofer-Salamon \cite[section 3.9]{EHS}.
In \cite{EHS} the compactness of moduli spaces when the contact manifold is compact is treated. We consider the non-compact manifold $\C\times Y$. If we show that for $\mathbf{J}$ admissible all holomorphic curves of finite energy are contained in the interior of $S(K\times Y)$ where $K \subset \C$ is a compact set, then we are in the same situation as in \cite{EHS} and the compactness follows from \cite[Theorem 3.9.1]{EHS}.

To see that the holomorphic curves under consideration here are contained in $S(K\times Y)$ it is enough to take $K = D_1$ from C2 in Definition \ref{Def_admJ} and use that the map $\pi: (S(\tilde{Y}), \tilde{J})\rightarrow (\C, i)$ is $(\tilde{J}, j)$-holomorphic on $S((\C \setminus K)\times Y)$. This implies that for any $\tilde{J}$-holomorphic curve $u:\Sigma\rightarrow S(\tilde{Y})$ the map $\pi\circ u$ is holomorphic and then its image is contained in some compact set $K$ by \cite[Lemma 4.2.1]{BiCo}.
 
%

The condition C1 of a cylindrical almost complex structure guarantees that no sequence of holomorphic strips can escape to the convex end, this in addition to the hypertight setting, guarantees no escape to the concave end.
Moreover, no sphere bubbling can happen since the symplectic manifold is exact, and disk bubbling is impossible on $L_{W}$ because it is exact and on $L_{P}$ because it is weakly exact or monotone with $N_{L}\geq 2$.  Then the only configurations that can appear in the compactified moduli spaces are broken trajectories.
\end{proof}


The differential is defined by 
 \begin{equation}
     d(\hat{x}_{-}) = \sum\limits_{\substack{|\hat{x}_{+}|= |\hat{x}_{-}|-1}} \sum\limits_{u\in \mathcal{M}(\hat{x}_{-}, \hat{x}_{+};\mathbf{J})} T^{\omega(u)}  \hat{x}_{+}
 \end{equation}

Once the compactness of the moduli spaces is ensured, it is standard to show that $d^{2}=0$. The reader can check for example \cite[Lema 3.2]{Fl1}.
Finally, we set $$CF(L_P,L_W;H,f,\mathbf{J}) := \oplus_{\star \in \pi_{0}(\mathcal{P}(\psi^1_H(\psi^{1}_{\tilde{f}}(L_P)), L_W))} CF_{\star}(L_P,L_W;H,f,\mathbf{J})$$
\begin{remark} In a similar way, the Floer complex $CF(\hat{P},L_{\Lambda}, H, J)$ for a pair $(P, \Lambda)$ consisting of a closed weakly exact or monotone pre-Lagrangian and a closed hypertight Legendrian of $(Y,\xi)$, can be defined. In this case $\hat{P},L_{\Lambda}\subset (S(Y),\omega)$, and since $Y$ is compact, the Floer complex consider here is a minor modification from the one defined in \cite{EHS}. The difference being that here we decided to consider the Novikov ring and all connected components of the path space. We will denote the homology of this complex by $HF(P,\Lambda)$.
\end{remark}
\begin{proposition}The homology of the complex $CF(L_P,L_W;H,f,\mathbf{J})$ is well defined and it is independent on the choice of compactly supported $H$ and generic and admissible $\mathbf{J}$ and depends on the sign of the function $f$ outside the compact set $B$. We denote this homology by $HF(L_P,L_W; [f])$.
\end{proposition}
\begin{proof}
The proof of this proposition follows the standard arguments. The only difference here is the non-compactness of $\tilde{Y}$ and then we only need to justify the compactness of the moduli spaces involved in the proofs. For the invariance under choice of compactly supported $H$ a chain map can be constructed using moving boundary conditions so that the moduli spaces of pseudo-holomorphic curves with moving boundary conditions will be compact by the choice of $\mathbf{J}$ in Definition \ref{Def_admJ}; C1 guarantees that no curves escape to the convex end, C3 that no curves escape to the $\C$ direction and the hypertight condition guaranties that no curve escape to the concave end. The invariance under change of almost complex structure follows from \cite[Theorem 3.7.3]{EHS}. For the invariance under change of function $f$, let  $f_{1}$ be another function satisfying choice 2 in \ref{chice} and with the same sign as $f$ outside $B$. To prove the invariance we use an auxiliary function $f_2$ so that $f_1$ and $f_2$ coincide inside the compact $B$
and so that $f$ and $f_2$ coincide outside a bigger compact, say $2B$, additionally $f_2$ should not
create any new intersection points. Then one can use a homotopy $g_\tau=\tau f_2 + (1-\tau)f$ and construct a chain map between the corresponding complexes using moving boundary conditions. And do similarly for $f_1$ and $f_2$. The coincidence of $f$ and $f_2$ outside $2B$ means that the moving boundary argument applies
(because the ends associated to $g_{\tau}$ remain constant to those of $f$ outside of $2B$).
The coincidence of $f_1$ and $f_2$ inside $B$ means that, by taking $B$ big enough, the complexes
of $f_1$ and $f_2$ coincide.
\end{proof}

\subsection{Invariance under non-compactly supported Hamiltonian perturbations} 

\subsubsection{Admissible Hamiltonian isotopies}
Let $H:[0,1]\times S(\tilde{Y})\rightarrow \R$ be a time dependent Hamiltonian and denote by $\psi^{t}$ its Hamiltonian flow.  We call $H$ admissible if: 
\begin{itemize}
 \item There is a constant $K>0$ such that $H$ has support on $[0,1]\times [-K, K] \times \tilde{Y}$.
\item There exist a compact set $B'$ containing the fixed compact set $B$ from \ref{chice}, $B\subset B'$, such that for all $t\in [0,1]$ and $q \in L_{V}\cap \pi^{-1}(\C\setminus B)$ we have  $\psi^{t}(q) \in L_{V}\cap\pi^{-1}(\C\setminus B')$, where $V\in\{P,W\}$.

\end{itemize}
\begin{remark}
An example of such a Hamiltonian isotopy is the Hamiltonian isotopy obtained by lifting the contact isotopy associated to a translation 
along the $x$-axis, by means of a suitable cutoff.
\end{remark}

\begin{proposition}\label{invarianza}
 The homology of the complex $CF(L_{P},L_ W; H, f, \textbf{J})$ is invariant under admissible Hamiltonian isotopies $\{\psi^{t}\}_{t\in[0,1]}$:
 $$HF(L_P,L_W;[f])\cong HF(\psi_{1}(L_P), L_W;[f]).$$

\end{proposition}

\begin{proof}
The proof of this statement is similar to the proof of the analogous statements for Lagrangian cobordisms in \cite[Proposition 4.3.1]{BiCo}. It consists in constructing a chain map using moving boundary conditions:
\begin{align*}
c:CF(L_{P},L_ W; H, f, \textbf{J})&\rightarrow CF(\psi^{1}(L_P), L_W; \psi^1\circ H\circ(\psi^1)^{-1},\psi^1\circ f\circ(\psi^1)^{-1},  \textbf{J}),\\
\hat{x} &\mapsto c(\hat{x})= \sum \limits_{\tilde{y}}\sum\limits_{u\in \mathcal{M}_{\beta}(\hat{x}, \hat{y}; \textbf{J})} T^{\omega(u) - \int\limits_{\R\times\{0\}} H_{\beta(s)}(u(s,0))ds}\hat{y},    
\end{align*}
where $\mathcal{M}_{\beta}(\hat{x}, \hat{y}; \textbf{J})$ denotes the following moduli space of $\textbf{J}$-holomorphic maps with moving boundary condition. 
Let $\beta: \R \rightarrow [0, 1]$ be a smooth function with $\beta(s) = 0$ for $s \leq 0$, $\beta(s) = 1$ for $s \geq 1$ and  $\dot{\beta}(s) > 0 \text{ on }(0,1)$. Thus, $u\in \mathcal{M}_{\beta}(\hat{x}, \hat{y}; \textbf{J})$ if:
\begin{enumerate}
\item $\overline{\partial}_{\textbf{J}}u =0,$
\item $u(s, 0) \in \psi^{\beta(s)(1-t)}(\psi^{1}_H(\psi^1_{\tilde{f}}((L_P)),$
\item $u(s,1) \in L_W,$
\item it has finite energy,
\item $\mu(\tilde{x})=\mu(u\#\tilde{y}).$
\end{enumerate}
The only difference with the setting in \cite[Proposition 4.3.1]{BiCo} is the possibility of losing compactness along the $\theta$ direction of the symplectization.
Notice that by conditions C1, C2 for admissible almost complex structure $\textbf{J}$, no $\textbf{J}$-holomorphic curve can escape to $\pm \infty$. Therefore the moduli spaces with moving boundary condition $ \mathcal{M}_{\beta}(\hat{x}, \hat{y}; \textbf{J})$ are compact. Once this is ensured, the proof of \cite[Proposition 4.3.1]{BiCo} adapts to this setting.
\end{proof}
\section{Proof of Theorem 1}

\begin{theorem}
Let $(Y,\xi)$ be a closed hypertight contact manifold and $\Lambda, \Lambda'\subset Y$ a pair of hypertight Legendrian submanifolds related by a connected hypertight Legendrian cobordism $(W;\Lambda, \Lambda')\subset (\C \times Y,\tilde{\xi})$.  If $P\subset Y$ is a closed weakly exact or monotone  with $N_P \geq 2$ pre-Lagrangian submanifold 
then $$HF(L_P,L_{\R\times \{0\}\times \Lambda};[f]) \cong  HF(L_P,L_{\R\times \{0\}\times \Lambda'};[f]).$$
\end{theorem}
\begin{proof}
The proof consists in first choosing special data $(H,f, \textbf{J})$ to define the Floer complex $CF(L_P,L_{W}; H, f, \textbf{J})$ and then finding suitable admissible Hamiltonian isotopy. 
Let $B \subset \C$ be such that $W$ is cylindrical outside $\pi_{\C}^{-1}(B)$, and $H$ a Hamiltonian with support on 
$[-K,K]\times B \times Y$ for $K>0$.
Let be $f(x,y) = \beta(x)ax$ where $a>0$ and for some $R>0$ such that $B\subset [-R, R]\times[-R,R]$ the smooth function $\beta:\R\rightarrow\R$ satisfies
$\beta(x) = 1$ on $(-\infty, -R]$, 
$\beta(x) = -1$ on $[R, \infty)$ 
and
$\dot{\beta}(x) <0$. 
Such an $f$ can be chosen such that $\pi_{\C}(L_P)$ looks like
\begin{center}
   \includegraphics[scale=0.70]{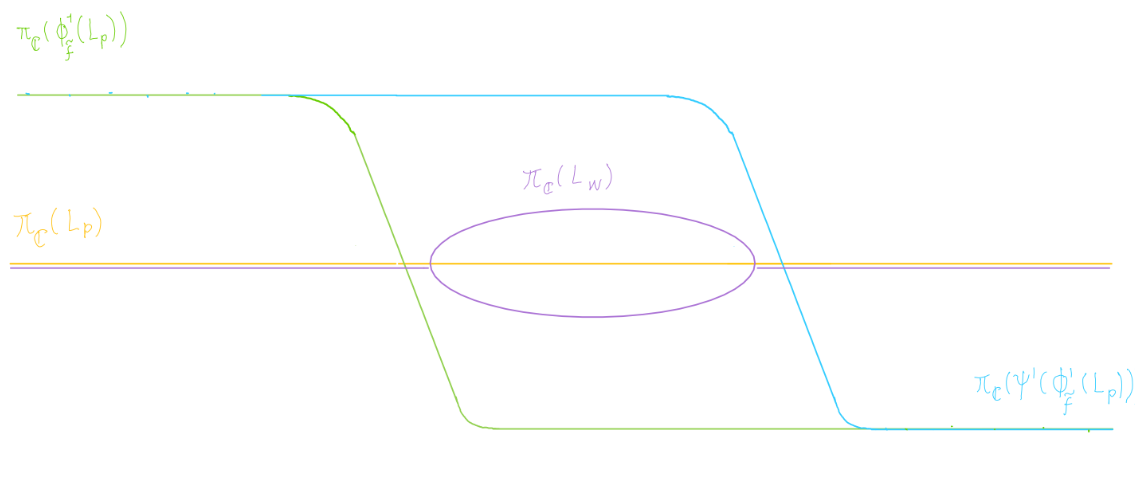}
\end{center}

The result follow from the fact that the translation on $x$, $T_{t}:\C \rightarrow \C$, $(x,y)\mapsto (x+t, y)$, induces a contact Hamiltonian on $\C\times Y$ that lifts to an admissible Hamiltonian isotopy, denoted by $\psi^t$, defined on a neighborhood of $L_P$ by the lift of the contact isotopy $\psi_{T_{t}\circ\pi_{\C}}$ and with support on a slightly bigger neighborhood.  Then from Proposition \ref{invarianza} 
$$HF(L_P,L_{W}; [f])\cong HF(\psi_{1}(L_P), L_W;[f]).$$
The right hand side of the equality is isomorphic to $HF(P,\Lambda)$ and the left hand side to $HF(P,\Lambda')$.
\end{proof}

\bibliographystyle{plain}
\bibliography{journal}
\end{document}